\documentclass{amsart}

\usepackage{amssymb}
\usepackage{amsfonts}
\input xy
\xyoption{all}
\usepackage{natbib} 

\newtheorem{theorem}{Theorem}[subsection]
\newtheorem{corollary}[theorem]{Corollary}
\newtheorem{definition}[theorem]{Definition}
\newtheorem{lemma}[theorem]{Lemma}

\newcommand{\smcat}[1]{\ensuremath{\mathcal{#1}}} %names a small category or functor.
\newcommand{\arrow}[3]{\ensuremath{#1\colon#2\rightarrow#3}} %draws an arrow
\newcommand{\marrow}[3]{\ensuremath{#1\colon#2\rightarrowtail#3}}
\newcommand{\eparrow}[3]{\ensuremath{#1\colon#2\twoheadrightarrow#3}}
\newcommand{\clarrow}[2]{\ensuremath{#1\rightarrow#2}}
\newcommand{\clmarrow}[2]{\ensuremath{#1\rightarrowtail#2}}
\newcommand{\adj}[4]{\ensuremath{#1\dashv#2\colon#3\rightarrow#4}}
\newcommand{\cprod}[2]{\ensuremath{#1\times #2}} %draws products closer
\newcommand{\cpull}[3]{\ensuremath{#1\times_{#2} #3}} %draws pullbacks 
\newdir{ >}{{}*!/-10pt/@{>}} %shifts an Xy-pic monic arrow tail one space by using @{ >->}

\mathchardef\mhyphen="2D

\title[Large structures founded on finite order arithmetic]{The large structures of Grothendieck founded on finite order arithmetic}
\author{Colin McLarty}

\begin{document}
\maketitle

Abstract: Such large-structure tools of cohomology as toposes and derived categories stay close to arithmetic in practice, yet existing foundations for them go beyond the strong set theory ZFC\@.  We formalize the practical insight by founding the theorems of EGA and SGA, plus derived categories, at the level of finite order arithmetic.  This is the weakest possible foundation for these tools since one elementary topos of sets with infinity is already this strong.

\section{Outline}

Grothendieck's unification of geometry and number theory led him to associate large structures to small ones.  For example each single arithmetic scheme has a large category of sheaves.  The point is not to study vastly many sheaves but to prove unifying theorems on schemes such as duality theorems.  For this Grothendieck posited \emph{universes} ``large enough that the habitual operations of set theory do not go outside'' them (SGA~1 VI.1~p.~146).  Some authors avoid the large structures, at least officially, because Zermelo Fraenkel set theory with choice (ZFC) cannot prove these universes exist.  But the structures reappear in citations and as motivation.  This paper removes the objection by proving the large structure theorems at the logical level of finite order arithmetic.

Finite order arithmetic \citep[Part II]{TakeutiProof}, or simple type theory with infinity, is $n$-th order arithmetic for all finite $n$.   It deals with numbers, sets of numbers, and sets of those, up through any fixed finite level.  Sections~\ref{S:MacLane}-- \ref{S:MCcategories} develop basic cohomology in any one of several set theories equivalent to this.

Sections~\ref{S:MTT}--\ref{S:CatsMTT} give a weak notion of a universe $\mathcal{U}$, and a simpler notion of $\mathcal{U}$-category than Grothendieck's (SGA~4 I.1.2), in a theory of classes  and collections conservative over set theory.  Section~\ref {S:largestructure} proves standard theorems on toposes, derived categories, and fibered categories.   This is the weakest possible level for Grothendieck's tools since a single elementary topos of sets with infinity is already as strong as finite order arithmetic.  Section~\ref{S:Fermat} relates this to proofs of  Fermat's Last Theorem.

\section{Set theory for basic cohomology}\label{S:MacLane}

Cohomology  needs a set $\mathbb{N}$ of natural numbers, and such rudimentary constructions as a product \cprod{A}{B} and union $A\cup B$ for any two sets $A,B$.  The delicate point is power sets.  Consider the finitely iterated power sets of $\mathbb{N}$:
  \[ \mathcal{P}^0(\mathbb{N})=\mathbb{N} \quad \mbox{ and }\quad\mathcal{P}^{n+1}(\mathbb{N})
           = \text{ the power set of } \mathcal{P}^{n}(\mathbb{N}).\]
Using replacement ZFC proves there is a set $\{\mathcal{P}^n(\mathbb{N}) |  n\in\mathbb{N}\}$ of all these, and vastly larger sets beyond.  Zermelo set theory using separation instead of replacement does not prove there is a set of all $\mathcal{P}^n(\mathbb{N})$ but it does prove all exist:
   \[\forall n\in \mathbb{N} \text{ there exists }\mathcal{P}^n(\mathbb{N}).\]
The axiom systems relevant here do not even prove that.   For each specified natural number, say 12, they prove the power set $\mathcal{P}^{12}(\mathbb{N})$ exists.  But they cannot prove the statement with quantifier $\forall n\in \mathbb{N}$ \citep{MathThe}.  

The separation axiom of Zermelo set theory says each formula $\phi(x)$ defines a subset $\{x\in y\ |\ \phi(x)\}\subseteq y$ of any set $y$.  Our set theories have \emph{bounded separation}, meaning the axiom only holds for formulas $\phi(x)$ where each quantifier has a bound $\forall u\in v$ or $\exists z\in w$.  In other words $\phi(x)$ must specify a set to look in for the values of each quantified variable.

For example an $I$-indexed set $\{X_i|i\in I\}$ here cannot be defined merely by giving a set $X_i$ for each $i\in I$, since lacking replacement there might be no set $X$ containing all the $X_i$.  We represent a set of disjoint sets $\{X_i|i\in I\}$ as a function \arrow{s}{X}{I} where for each $i\in I$ the set $X_i$ is the pre-image $s^{-1}(i)\subseteq X$. 

All this can be formalized in finite order arithmetic but set theoretic language is more convenient here.  Suitable set theories include the elementary theory of the category of sets (ETCS) \citep{LawElem}, and the fragment of ZFC without replacement or foundation and with separation only for bounded formulas.  This fragment is Mathias's ZBQC minus foundation or his \textbf{Mac} minus foundation and transitive containment.  Equiconsistency of finite order arithmetic and all these named set theories follows from \citet{MathThe}.  We use MC to suggest ``Mac~Lane set theory''  as shorthand for any of these set theories.

\section{Basic cohomology in MC}\label{S:MCcategories}
The category  $Set^{\smcat{C}}$   of all set-valued functors on a small category \smcat{C} is not a set.  But each set-valued functor  can be defined as one set so that $Set^{\smcat{C}}$  is a definable class.  Sections~\ref{SS:Smcats}--\ref{SS:Yoneda} show bounded separation proves results such as the Yoneda lemma.  Sections~\ref{SS:topologies}--\ref{SS:fundgroup} discuss issues from SGA and EGA.  Section~\ref{SS:injectives} uses bounded separation to give infinite injective resolutions of sheaves of modules, where previous published proofs use at least countable replacement.  

To be precise about Theorem~\ref{T:Topos}, sheaves over a site do not form a model of the topos axioms in MC, because they do not form one set.  Rather MC proves each topos axiom, and thus each theorem of elementary topos theory, when sheaves and natural transformations over a given site are taken as objects and arrows.

Contrast Theorem~\ref{T:GRTopos}:  In the theory MTT of classes and collections, each Grothendieck topos exists as a single class and is a model of the topos axioms.

\subsection{Small categories}\label{SS:Smcats}

A \emph{small category} \smcat{C} is a set $C_0$ called the objects and a set $C_1$ called the arrows with domain and codomain functions $d_0, d_1$ and composition $m$ satisfying the category axioms.  A \emph{functor} \smcat{\arrow{F}{C}{D}} of small categories is an \emph{object part} \arrow{F_0}{C_0}{D_0} and \emph{arrow part} \arrow{F_1}{C_1}{D_1} preserving composition and identity arrows.  In fact $\mathcal{F}$ is fully determined by its arrow part $F_1$. 

For any small categories \smcat{B,C} there is a small category \smcat{B^C} of all functors \clarrow{\smcat{C}}{\smcat{B}}, with natural transformations as arrows~\citep[pp.~40--42]{CfWM}.  If functors are represented by their arrow parts then the set of all functors appears as a subset of the function set $B_1^{C_1}$.  Natural transformations are certain functions \clarrow{C_0}{B_1} from objects of \smcat{C} to arrows of \smcat{B}, so they form a subset of the function set $B_1^{C_0}$. 

\subsection{Presheaves}\label{SS:Presheaves}
A \emph{presheaf} $F$ on a small category \smcat{C} is a contravariant functor from \smcat{C} to sets.   But sets do not form a small category.  So $F$ is defined as a $C_0$-indexed set of sets \arrow{\gamma_0}{F_0}{C_0} and an action \arrow{e_F}{F_1}{F_0} where
     \[F_1 = \{\langle s,f \rangle \in \cprod{F_0}{C_1}\ |\ \gamma_0 (s)=d_1(f)  \}\]
For each $A\in C_0$ define the value $F(A)$ to be:  
     \[F(A) = \gamma_0^{-1}(A)= \{s\in F_0\ |\ \gamma_0(s) = A \}.\]
We require:
\begin{enumerate}
   \item  For all arrows \arrow{g}{B}{A} in \smcat{C}, if $s\in F(A)$ then $e_F\langle s,g \rangle\in F(B)$.
   \item   If $s\in F(A)$ then  $e_F\langle s,1_A \rangle = s$ for the identity arrow $1_A$.
   \item   For any \arrow{h}{C}{B} in \smcat{C}, $e_F\langle s, gh \rangle = e_F\langle e_F\langle s,g\rangle,h\rangle$.
\end{enumerate}
By Clause~1, \arrow{F(g)}{F(A)}{F(B)} is defined by $(F(g))(s)=e_F\langle s,g \rangle$.  Clauses~2--3 express functoriality for identity arrows and composition.  

A technical lemma will be used later:

\begin{lemma}\label{L:limitofpresheaf} Every presheaf  $F$ on a small category \smcat{C} has a limit set $\varprojlim F$ and colimit set $\varinjlim F$.
\end{lemma}
\begin{proof}The proofs of \citet[pp.~110 and 112~ex.~8]{CfWM} work in MC.  Note Mac~Lane construes elements of  $\varprojlim F$ as certain functions from $C_0$ to $F_0$.
\end{proof}

A \emph{natural transformation} \arrow{\eta}{F}{G} of presheaves is a function over $C_0$
\[\xymatrix@R=.1pc@C=1pc{ F_0  \ar[rr]^{\eta}  \ar[ddr]_{\gamma_0} && G_0    
         \ar[ddl]^{\gamma_0'}\\  &&&   \gamma_0 =  \gamma_0'\eta  \\  & C_0  }\]
which commutes with the actions $e_F$ and $e_G$ in the obvious way.

Rougly speaking, prescheaves on \smcat{C} form a locally small, complete, and cocomplete category.   That will be exactly true in the class and collection theory MTT.  But the set theory MC requires more cautious statements as follows. 

Any set of presheaves has a set of all transformations between them.  That is, an \emph{$I$-indexed set} of presheaves on a small category \smcat{C} is a \cprod{C_0}{I}-indexed set of sets \arrow{\gamma_0}{F_0}{\cprod{C_0}{I}} with an $I$-indexed action \arrow{e_F}{F_1}{F_0} where now
  \[F_1 = \{\langle s,f,i\rangle \in \cprod{F_0}{\cprod{C_1}{I}}\ |\ \gamma_0 (s)=          
               \langle d_1(f),i \rangle  \}\]
Each $A\in \smcat{C}$ and $i\in I$ determine a set $F(A,i)$.  The action must satisfy equations saying for each arrow \arrow{g}{B}{A} in \smcat{C} and index $i$ it induces a function  \arrow{F(g,i)}{F(A,i)}{F(B,i)} and is functorial.  For any  $i,j\in I$ a natural transformation \clarrow{F(\_,i)}{F(\_,j)} is a subset of \cprod{F_0}{F_0}.  So all these transformations form a subset of the powerset $\mathcal{P}(\cprod{F_0}{F_0})$, with defining conditions bounded by $F_1$. 

Given parallel natural transformations \arrow{\eta,\iota}{F}{G} of presheaves the usual constructions of a presheaf equalizer and a presehaf coequalizer work in MC~\cite[p.~115]{CfWM}. And every indexed set \arrow{\gamma_0}{F_0}{\cprod{C_0}{I}} of presheaves has a coproduct $\coprod F$ given by projection to $C_0$:
  \[\xymatrix@C=1pc{ \coprod F=F_0 \ar[rr]^<<<<<<{\coprod \gamma_0} && C_0 & = & F_0 \ar[rr]^<<<<{\gamma_0} && \cprod{C_0}{I} \ar[rr]^<<<<<{p_0} && C_0} \]
For each $A$, the value $\coprod F(A)$ is the disjoint union of all $F(A,i)$ for $i\in I$.   So the action \arrow{e_F}{F_1}{F_0} is also the action for $\coprod F$.  The usual construction of a product of a set of presheaves also works in MC, using the function set $F^I$.

\subsection{The Yoneda lemma}\label{SS:Yoneda}

Each object $B$ of a small category \smcat{C} \emph{represents} a presheaf $R_B$ assigning to each object $A$ of \smcat{C} the set 
  \[R_B(A) = \mathrm{Hom}_{\smcat{C}}(A,B)\] 
of all arrows from $A$ to $B$. Each \smcat{C} arrow \arrow{f}{A'}{A} gives a function \[R_B(f) :  \mathrm{Hom}_{\smcat{C}}(A,B) \rightarrow \mathrm{Hom}_{\smcat{C}}(A',B)\] defined by $R_B(f)(g)=gf$.   There is even a $C_0$-indexed family of all representable presheaves $R_B$, namely $C_1$ with the domain and codomain functions:
\[\xymatrix@C=4pc{ C_1 \ar[r]^<<<<<<<<{\langle d_0,d_1 \rangle} & \cprod{C_0}{C_0} } \]

Any arrow \arrow{h}{B}{D} of \smcat{C} induces a natural transformation of presheaves in the same direction, defined in the natural way:
  \[ \arrow{R_h}{R_B}{R_D} \qquad R_h(g)=hg \quad \text{for all } g\in R_B \]
This operation is functorial in that $R_hR_k=R_{hk}$ and $R_{(1_B)}=1_{(R_B)}$.

The simplest \emph{Yoneda lemma} says for any presheaf $F$ on \smcat{C} and object $B$ of \smcat{C}, natural transformations \clarrow{R_B}{F} correspond naturally to elements of $F(B)$.  \citet[p.~59]{CfWM} has a proof suitable for MC\@.  So the representables are \emph{generators}: any two distinct natural transformations of presheaves \arrow{\eta\neq \theta}{F}{G} are distinguished by some natural transformation \arrow{\nu}{R_B}{F}.
  \[ \xymatrix{ R_B \ar[r]^{\nu} & F \ar@<1ex>[r]^{\eta} \ar@<-1ex>[r]_{\theta} & G & \eta\nu\neq\theta\nu } \] 

A stronger Yoneda lemma says every presheaf is a colimit of presheaves $R_B$.   The elementary proof by \citet[p.~51]{JTop} is easily formalized in MC\@.

\subsection{Topologies}\label{SS:topologies}
A \emph{Grothendieck topology} $J$ on a small category \smcat{C} assigns each object $A$ of \smcat{C} a set of sets of arrows to $A$ called the set of \emph{covers} of $A$.  So it is a $C_0$-indexed set of sets of arrows subject to familiar conditions all bounded by $C_1$ and its powerset.  Thus there is a set of all topologies on \smcat{C}.  

A $J$-\emph{sheaf} on $\langle \smcat{C},J\rangle$ is a presheaf meeting a compatibility condition: for every $J$-covering family $\{\arrow{f_i}{A_i}{A} | i\in I\}$ the value $F(A)$ is an equalizer
  \[ \xymatrix{ F(A) \ar[r]^<<<<{\nu} &  \prod_{i} F(A_i)  \ar@<1ex>[r]^<<<<{\eta} \ar@<-1ex>[r]_<<<<{\theta} & \prod_{i, j} F(\cpull{A_i}{A}{A_j}) } \] 

The usual proofs work in MC to show every presheaf $F$ on a site $\langle \smcat{C},J\rangle$ has an \emph{associated sheaf} $\mathbf{a}F$ and natural transformation \arrow{i}{F}{\mathbf{a}F} such that every natural transformation \arrow{\eta}{F}{S} to a $J$-sheaf $S$ factors uniquely through $i$.  This universal property shows each natural transformation of presheaves \arrow{\theta}{F}{G} induces a natural transformation of the $J$-sheaves \arrow{\mathbf{a}\theta}{\mathbf{a}F}{\mathbf{a}G}.

\begin{theorem}\label{T:Topos}All theorems of elementary topos theory hold for sheaves over any site in MC.  See for example \citep{JTop}.
\end{theorem}
\begin{proof}The elementary topos axioms and proofs involve only bounded constructions on objects and arrows.  
\end{proof}

As MC does not have the replacement axiom it takes a little care to show:

\begin{lemma}\label{L:sheafbound} Any set of presheaves $\{F_i|i\in I\}$ over a site $\langle \smcat{C},J\rangle$ has a set of associated sheaves $\{\mathbf{a}F_i|i\in I\}$.  
\end{lemma} 
\begin{proof}Following \citet[p.~129]{MacMoer} it suffices to show this for the operator $F\mapsto F^+$ on presheaves  in place of $F\mapsto \mathbf{a}F$, since $\mathbf{a}F=F^{++}$.  

For each object $A$ of \smcat{C} the set $F^+(A)$ is a set of equivalence classes of compatible families of sections on covers of $A$.  In the notation of Section~\ref{SS:Presheaves}, a family of sections is a subset of $F_1$ so every $F^+(A)$ is a subset of the iterated power set $\mathcal{P}^2(F_1)$.  So $(F^+)_0$ is also a subset of $\mathcal{P}^2(F_1)$, and the structure map \arrow{\gamma}{(F^+)_0}{C_0} has graph a subset of  \cprod{\mathcal{P}^2(F_1)}{C_0}.  The quantifiers defining this subset are bounded by $C_0$, $C_1$, and $F_0$.  Analogous treatment works for the action on $F^+$.  

Now suppose given a set $\{F_i|i\in I\}$.  This is technically a \cprod{C_0}{I}-indexed set of sets \arrow{\gamma_0}{F_0}{\cprod{C_0}{I}} with  $I$-indexed action \arrow{e_F}{F_1}{F_0} on the set
  \[F_1 = \{\langle s,f,i\rangle \in \cprod{F_0}{\cprod{C_1}{I}}\ |\ \gamma_0 (s)=          
               \langle d_1(f),i \rangle  \}\]
The set $\{\mathbf{a}F_i|i\in I\}$ is formed in the single  iterated power set  $\mathcal{P}^2(F_1)$ for this $F_1$.  
\end{proof}

\subsection{Size of sites}\label{SSTechnicalsites} Most textbooks and published proofs make number theoretic sites proper classes.  The issue is not \emph{gros} versus \emph{petit} sites.  Those do not differ in set theoretic size but in the geometric ``size'' of fibers.   Fibers may have any dimension in a gros site but are 0-dimensional in a petit site.  The issue is that replacing proper class sites by sets is not trivial.  The comparison lemma, SGA~4 VII.3.3, our Theorem~\ref{Th:Compare}, works for some cases.   Verdier SGA~4 III.0 notes the use of this lemma ``obliges us to certain contortions.''

As another strategy, publications often use scheme sites local on the fiber so the site is closed under all set-sized disjoint unions \clarrow{\coprod_iY_i}{X} and it cannot be a set itself.  Often the maps could be limited to the \emph{quasi-compact} so only finite unions arise.  See EGA~I~6.3.1 or \citet[p.~90]{TamEt}.   

In any case Sections~\ref{S:MTT}---\ref{S:largestructure} handle arbitrary large sites in a theory of classes and  collections no stronger logically than the set theories we call MC.

\subsection{Functoriality of presheaves and sheaves}
 Grothendieck and Verdier SGA~4 I.5, and Verdier SGA~4 III.1--3 prove various relations between small site functors, and presheaf and sheaf functors.  All are all provable in MC\@.  Suitable bounds are implicit in SGA, and many are explicit in \citet[Ch.~2]{JTop}.

The basic case is composing a presheaf $F$ on \smcat{D} with a functor \arrow{u}{\smcat{C}}{\smcat{D}} to get a presheaf $u^*F$ on \smcat{C}.  If $F$ has object part $F_0=\{F_A\ |\ A\in \smcat{D}_0\}$ then $(u^*F)_0=\{F_{u(B)}\ |\ B\in \smcat{C}_0\}$.   The action of the arrows $C_1$ on $u^*F$ is the action of $D_1$ on $F$ composed with the arrow part \arrow{U_1}{C_1}{D_1} of $u$.

A natural transformation \arrow{\eta}{F}{G} of presheaves on \smcat{D} is a $D_0$ indexed set of functions, and its image  \arrow{u^*\eta}{u^*F}{u^*G} is the same only re-indexed over $C_0$.  Formally it is a pullback: 
     \[\xymatrix@R=.1pc@C=1pc{  &&  (u^*G)_0  \ar[rrr]  \ar[ddddddl]|<<<<<<<<\hole   &&& G_0    
                                         \ar[ddddddl]^{\gamma_0'} \\  \\
                      (u^*F)_0  \ar[rruu]^{u^*\eta} \ar[rrr]   \ar[ddddr]  &&& F_0  \ar[rruu]^<<<<<{\eta}  \ar[ddddr]_{\gamma_0}  \\   
                         \\  \\  \\  &  C_0 \ar[rrr]^{u_0} &&& D_0  }\]

The set theory MC makes $u^*$ a well defined functorial operation on presheaves, but not actually a functor, because MC has no actual category of all presheaves over a site.  MC does prove $u^*$ has well defined left and right adjoint functorial operations defined by lightly adapting either SGA~4 I.5 and~4 III.1--3, or the relevant parts of  \citet{JTop}.  And it verifies the corresponding functors and adjunctions for sheaves on sites $\langle \smcat{C},J\rangle$.  We spell out a case used in Section~\ref{SS:oversites}:

\begin{lemma}\label{L:setstosets}For any functor of small categories \arrow{u}{\smcat{C}}{\smcat{C}'} the presheaf operator $u_*$ right adjoint to $u^*$ takes each set $\{F_i|i\in I\}$ of presheaves on \smcat{C} to a set of presheaves $\{u_*(F_i)| i\in I\}$ on \smcat{C'}.
\end{lemma}
\begin{proof}  For each object $Y$ of \smcat{C'}, Grothendieck and Verdier (SGA~4 I.5) form a small category $\mathcal{I}^u_{Y}$, now more often written $(u\downarrow Y)$ the comma category of $u$ over $Y$.  Objects of $(u\downarrow Y)$ are pairs $\langle X,m\rangle$ with  \arrow{m}{u(X)}{Y}, and arrows \arrow{\xi}{\langle X,m\rangle}{\langle X',m'\rangle} are defined by commutative triangles 
\[\xymatrix@R=.5pc@C=1pc{   X \ar[dd]_{\xi} &  u(X) \ar[drrr]^m   \ar[dd]_{u(\xi)} \\       &&&& Y. \\   x' & u(X') \ar[urrr]_{m'}}\]
A projection functor \arrow{pr_Y}{(u\downarrow Y)}{\smcat{C}} takes $\langle X,m\rangle$ to $X$.  

Composing $pr_Y$ with any presheaf $F$ on \smcat{C} gives a presheaf $pr_Y^*F$.  Define $(u_*F)(Y)$ as $\varprojlim pr_Y^*F$ (Lemma~\ref{L:limitofpresheaf}).  Every \arrow{f}{Y}{Y'} in \smcat{C'} induces a functor $(u\downarrow f)$ from $(u\downarrow Y)$ to $(u\downarrow Y')$, and so a function \arrow{(u_*F)(f)}{(u_*F)(Y')}{(u_*F)(Y)}.  

Each set $u_*F(Y)$ can be taken as a subset of the function set $F_0^{\cprod{C_0}{C'_1}}$.  So the graph of the structure map \arrow{\gamma}{(u_*F)_0}{C'_0} is a subset of \cprod{F_0^{\cprod{C_0}{C'_1}}}{C'_0}.  The quantifiers defining this subset are bounded by $C_0$, $C_1$, $C'_0$, $C'_1$, and $F_0$.  Analogous treatment works for the action, so $u_*F$ is a presheaf on \smcat{C'}.

Now suppose given a set  $\{F_i|i\in I\}$, that is an indexed set  \arrow{\gamma_0}{F_0}{\cprod{C_0}{I}} with  $I$-indexed action.  Each $u_*(F_i)$ can be constructed this way, but all working in the single set \cprod{F_0^{\cprod{C_0}{C'_1}}}{C'_0} for this set $F_0$, so as to define a single set  $\{u_*(F_i)| i\in I\}$.
\end{proof}

\subsection{\'Etale fundamental groups}\label{SS:fundgroup}
A topological space $X$ has \emph{covering spaces} as e.g.~a helix covers a circle.  Symmetries of a suitable cover of $X$ form its (topological) fundamental group, like a Galois group, and reveal much about $X$. The \emph{finite \'etale covers} of a scheme $X$ and the corresponding \'etale fundamental group give uncannily good analogues to topological covering spaces and include Galois groups as special cases~\citep{SGA1}.

The theory of finite \'etale covers is elementary algebra (EGA~IV).  The fundamental group uses a category of ``all'' finite \'etale covers of a scheme $X$, but that means all \emph{up to isomorphism}.  Since these covers are given by finitely generated extensions of coordinate rings on $X$,  MC provides a sufficient category using the set of all extensions generated by finite subsets of some fixed infinite set $G$.

\subsection{Injectives and cohomology groups}\label{SS:injectives}

\citet{BaerAb} used replacement to prove every module embeds in an injective module.  \cite{EckmannSchopf} proved it without replacement, but requiring choice to show divisible Abelian groups are injective~\citep{BlassInj}.  \citet{Tohoku} adapted Baer's proof to sheaves of modules on topological spaces in a way that actually works in any Grothendieck topos.   Because it uses choice the Eckmann-Schopf  proof does not  lift directly to the topological case let alone all Grothendieck toposes. \citet{BarrPoints} overcame this by showing every Grothendieck topos $E$ has \emph{Barr covers} satisfying choice.  

MC suffices to formalize that proof, for sheaves of modules over any site, not only for single injective embeddings but for infinite injective resolutions. 

\subsubsection{Resolutions in sets}\label{SS:resolutionssets}
Standard proofs work in MC to show every Abelian group embeds in a divisible one, and every divisible Abelian group is injective.  That is all there is to know about injective resolution of Abelian groups, since quotients of divisible groups are divisible, so every embedding of an Abelian group $A$ into a divisible $I_0$ gives a length one injective resolution:
                \[ \xymatrix{A \ar@{ >->}[r]  &   I_0 \ar[r]   &  I_0/A  \ar[r]  &  0 }.  \]

Injective modules over a ring are more subtle, and can require infinite resolutions.  We use a result which serves again in Section~\ref{SS:oversites}, first given by \citet[p.~108]{MarandaInj}  and Verdier \citep[\S V  lemma 0.2]{SGA4mimeo}:  

\begin{lemma}\label{L:injectives}If a functor \smcat{\arrow{F}{B}{A}} has a left exact left adjoint \smcat{\arrow{G}{A}{B}} with monic unit and each object in \smcat{B} embeds in an injective then so does each in \smcat{A}.
\end{lemma}
\begin{proof} If units are monic, every monic \clmarrow{\smcat{G}(A)}{B} has monic adjunct \clmarrow{A}{\smcat{F}(B)}.  Since \smcat{G} preserves monics, \smcat{F} preserves injectives.  If object $A$ in \smcat{A} has a monic \clmarrow{\smcat{G}(A)}{I} to an injective in \smcat{B}, the adjunct \clmarrow{A}{\smcat{F}(I)} is monic.
\end {proof}

\begin{corollary}\label{C:moduleinjectives}For any ring $R$, every $R$-module embeds in an injective.
\end{corollary}
\begin{proof} Let \smcat{F} take each Abelian group $A$ to the $R$-module $Hom_{\mathbb{Z}}(R,A)$ of additive functions from $R$ to $A$, with $r\cdot f$ defined by $(r\cdot f) (x) = f(r\cdot x)$.  It has left exact left adjoint \smcat{G} the underlying group functor.  For each $M$ the unit $\eta_M$ takes each $m\in M$ to the function $r\mapsto r\cdot m$, so is monic.
\end{proof}

Lemma~\ref{L:injectives} and Corollary~\ref{C:moduleinjectives} give a procedure to produce injective resolutions of any finite length $n$ for any module $M$.  That is exact sequences 
   \[ \xymatrix{M \ar@{ >->}[r]  &   I_0 \ar[r]  & \cdots \ar[r] &  I_n  }  \] 
with all $I_i$ injective.  Define sequences $I_i$ and $M_i$ inductively:
\begin{enumerate}
  \item Set $M_0=M$.
  \item\label{I:divisible} Embed $M_i$ as an additive group into a divisible group \clmarrow{M_i}{M_{di}}.
  \item Form the injective $R$-module $I_i=Hom_{\mathbb{Z}}(R,M_{di})$ with monic \clmarrow{M_i}{I_i}.
  \item Start again, with the quotient $M_{i+1}=I_i/M_i$.
\end{enumerate} 

Textbooks immediately conclude there are infinite injective resolutions, by implicit use of (countable) replacement.  In fact MC also proves that conclusion, but only after altering the procedure to bound it inside one set for each module $M$.

The bound will be the function set $M^{\cprod{\mathbb{Z}}{R^{\mathbb{N}}}}$ which has an $R$-module structure induced by $M$.  Here $R^{\mathbb{N}}$ is the set of infinite sequences in $R$. Say a function \arrow{f}{\cprod{\mathbb{Z}}{R^{\mathbb{N}}}}{M} is \emph{cut off at} $n\in\mathbb{N}$  if $f(m,\sigma)=0$ for every sequence $\sigma$ which does \emph{not} have $\sigma(i)=0$ for all $i\geq n$.  In effect a function cut off at $n$ is an element of $M^{\cprod{\mathbb{Z}}{R^n}}$.  So, a function cut off at $n+1$ can also be regarded as a function from $R$ to the set $M^{\cprod{\mathbb{Z}}{R^n}}$ of functions cut off at $n$.

Also,  notice Step~\ref{I:divisible} is idle for $i\geq 1$ since all $I_i$ and all $M_{i+1}=I_{i+1}/I_i$ are divisible groups.  So it suffices to give an infinite injective resolution for each module $M$ with divisible underlying group.  For this case $M_i=M_{di}$ for all $i\in\mathbb{N}$. 

For any ring $R$, and $R$-module $M$ with divisible underlying group, define this induction parallel to the one above:

\begin{enumerate}
  \item[(1\textprime)] Let the subset $N_0\subset M^{R^{\mathbb{N}}}$ contain just the additive functions cut off at $0$. In effect these are additive functions       
                    \clarrow{\mathbb{Z}}{M}, so $N_0\cong M$.
   \item[(1\textprime\textprime)] Define equivalence relation $E_0$ as the identity on $N_0$.  The point is \[M\cong N_0\cong N_0/E_0.\]
  \item[(3\textprime)] Given the subset $N_i\subset M^{R^{\mathbb{N}}}$ with every function cut off at $i$, and equivalence relation $E_i$ on it, define a certain subset $J_i\subset M^{R^{\mathbb{N}}}$ of functions which are cut off at $i+1$.  Namely, think of these as functions \clarrow{R}{M^{\cprod{\mathbb{Z}}{R^n}}}.  Let $J{i}$ contain just those whose values all lie in $N_i$ and which are additive when seen as functions \clarrow{R}{N_i/E_i}.   Let $Q_I$ be the pointwise equivalence relation making functions  \clarrow{R}{N_i} equivalent iff they are equal as functions  \clarrow{R}{N_i/E_i}.
   \item[(3\textprime\textprime)] There is a natural monic \marrow{h}{N_i}{J_i} where for each $g\in N_i$ the value $h(g)$ is the unique $R$-linear function  \clarrow{R}{N_i/E_i} taking $1\in R$ to $g$.  
     \item[(4\textprime)] Define $N_{i+1}=J_i$ with $E_{i+1}$ the smallest equivalence relation containing both $Q_i$ and the relation induced by the submodule   \marrow{h}{N_i}{J_i}.
\end{enumerate}

For every $i\in \mathbb{N}$ the quotient $N_i/E_i$ is isomorphic as $R$-module to the module $M_i$ above, while each $J_i/Q_i$ is isomorphic to $I_i$ above,  So this gives an isomorphic copy of the resolution by $I_i$ above.  Bounded separation suffices to show this infinite resolution is one set, since it is all bounded by$M^{R^{\mathbb{N}}}$.

\subsubsection{Resolutions over sites}\label{SS:oversites}

Now let $E$ be any site, and $R$ any sheaf of rings on it.  We want to show sheaves of modules on $R$ have infinite injective resolutions.  The argument of Section~\ref{SS:resolutionssets} works in any elementary topos with natural numbers and choice, so it works for sheaves over any site whose sheaves satisfy choice in the obvious way: every sheaf epimorphism has a right inverse.  So it works over any Barr cover of $E$~\citep[p.~261]{JTop}.  Compare \cite{VanOsdol}.   We have only to show in MC that every site has a Barr covering site and each infinite resolution descends (as a single set) along that Barr cover.  The first is clear from the construction by \citet[p.~511--13]{MacMoer}.

\begin{corollary}\label{C:moduleinjectivessites} For any surjection of ringed toposes \adj{f^*}{f_*}{(\smcat{B},R')}{(\smcat{A},R)} if every $R'$ module embeds in an injective then so does every $R$ module, and $f_*$ preserves injectives.
\end{corollary}
\begin{proof} Lemma~\ref{L:injectives}, noting topos surjections have monic unit.
\end{proof}

\begin{lemma}\label{L:directpreserves} For any geometric morphism \adj{f^*}{f_*}{\smcat{B}}{\smcat{A}} where \smcat{B} satisfies axiom of choice, $f_*$ preserves all exact sequences of modules over any ring. 
\end{lemma}
\begin{proof} Direct image functors preserve monics.  In the choice topos \smcat{B} every quotient \eparrow{q}{M}{M/J} has a right inverse function \clarrow{M/J}{M}  (generally not a homomorphism), so $f_*(q)$ also does, so $f_*(q)$ is an epimorphism and thus a quotient.  
\end{proof}

\begin{theorem} For any sheaf of rings $R$ on any site $\langle \smcat{C},J\rangle$, every sheaf of $R$-modules $M$ has an infinite injective sheaf resolution.
\end{theorem} 
\begin{proof} Over any Barr cover of  $\langle \smcat{C},J\rangle$, $f^*(M)$ has an infinite injective resolution existing as a single set
  \[ \xymatrix{f^*(M) \ar@{ >->}[r]  &   I_0 \ar[r]  & \cdots \ar[r] &  I_n  \ar[r] & \cdots }  \]
By Lemma~\ref{L:directpreserves} its $f_*$ image is exact and since the unit \clmarrow{M}{f_*f^*(M)} is monic this is an injective resolution:
  \[ \xymatrix{M  \ar@{ >->}[r]  &  f_*(I_0) \ar[r]  & \cdots \ar[r] &  f_*(I_n)  \ar[r] & \cdots }  \]
By  Lemmas~\ref{L:sheafbound}  and~\ref{L:setstosets} this resolution exists as a single set. 
\end{proof}

\subsubsection{Cohomology groups} So MC proves every sheaf of modules has an infinite injective resolution, existing as one set.  Indeed it can define a specific resolution for any module over a given site.  The axiom of choice in MC is used to verify the construction works, specifically by showing divisible groups and Barr covers have the requisite properties,  but choice is not used to specify the resolution.  The usual formalities of homological algebra show cohomology groups are functorial, exact, and effaceable.  So MC can specify a long exact cohomology sequence for each short exact sequence of sheaves of modules.   Standard results on \v{C}ech cohomology and spectral sequences also follow.

\subsubsection{Resolutions at the level of 3rd order arithmetic}\label{S:Third}
Part of the above works in a the much weaker level of 3rd order arithmetic.   ZFC[0] is ZFC without powerset, while ZFC[1] extends that by positing the natural numbers have a powerset.  ZFC[0] has the proof theoretic strength of second order arithmetic, and ZFC[1] of third order.  See \cite{ZbierskiHigher}.

ZFC[0] proves sets have cartesian products \cprod{A}{B}, every $A$ has a set $\mathrm{Fin}(A)$ of all finite subsets, and every equivalence relation on a set has a set of equivalence classes.  So every set generates a free Abelian group, and tensor products exist.  ZFC[1] proves all countable set have power sets, so countable $A,B$ have a set of all functions \clarrow{A}{B}.  And these set theories have the axiom of replacement.

\begin{corollary} Provably in \emph{ZFC[1]}: every countable module $M$ on a countable ring $R$ has an infinite injective resolution. 
\end{corollary} 
\begin{proof} For any $n\in\mathbb{N}$, the countable product $R^n$ exists provably in ZFC[1].  In the proof of Section~\ref{SS:resolutionssets} replace $R^{\mathbb{N}}$ by $R^n$ to construct a finite resolution.
  \[ \xymatrix{M \ar@{ >->}[r]  &   I_0 \ar[r]  & \cdots \ar[r] &  I_n  }  \]  
By construction this is isomorphic to the initial segment of the longer sequence using $R^{n+1}$.  For each $n\in \mathbb{N}$ specify $I_n$ by the $n$-length resolution,  and define \clmarrow{I_n}{I_{n+1}} using the canonical isomorphism of this $I_n$ to the $n$-th term of the $(n+1)$-length resolution.  By replacement, this provides an infinite resolution.  
\end{proof}

So ZFC[1] provides all the standard long exact cohomology sequences for ordinary countable modules.  This is the core of cohomological number theory.  But existing proofs use cohomology beyond this core, and use  more techniques than cohomology.  Formalizing them in low order arithmetic will take further analysis.

\section{Classes of sets and collections of classes} \label{S:MTT} 

Take the sets as one type and add classes of sets as a higher type and collections of classes as another.  We indicate sets by italics $x,A$, classes by calligraphic $\mathcal{A,B}$, and collections by fraktur $\mathfrak{A,B}$.  As above, $x\in B$ or $A\in B$ say a set $x$ or $A$ is in set $B$.  Use $A\in_1 \mathcal{A}$ to say set $A$ is in class $\mathcal{A}$, and $\mathcal{A}\in_2 \mathfrak{B}$ to say class $\mathcal{A}$ is in collection $\mathfrak{B}$.  Any of the set theories we abbreviate as MC, plus the higher type axioms and inference rules below, give \emph{Mac~Lane Type Theory} (MTT).  

Gentzen-style cut elimination shows MTT is conservative over the set theory.  By using only set theoretic abstracts we avoid the more complicated cut elimination for full simple type theory.  See Takeuti's analogous proofs for Peano arithmetic in place of set theory   \citeyearpar[pp.~77f.]{TakeutiConserv}\@ and for set theory \citeyearpar[p.~176]{TakeutiProof}.

A \emph{set theoretic formula} is a formula which may include variables over classes and collections but has quantifiers only over sets. So class inclusion is set theoretic:
\[\mathcal{A}\subseteq_1 \mathcal{B}\ \leftrightarrow\ \forall x\, (x\in_1\mathcal{A} \rightarrow x\in_1 \mathcal{B})\]
Inclusion of collections is well defined, expressed by a formula of MTT
\[\mathfrak{A}\subseteq_2 \mathfrak{B}\ \leftrightarrow\ \forall \mathcal{X}\,(\mathcal{X}\in_2\mathfrak{A} \rightarrow \mathcal{X}\in_2 \mathfrak{B})\]
Collection inclusion can be used in theorems and proofs of MTT\@.  It \emph{cannot} be used to define classes or collections.  It is not set theoretic.  It quantifies over classes.

The key device is \emph{set theoretic abstracts} where set theoretic formulas define classes and collections.  For example, since an ordered pair of sets is a set, a set theoretic abstract defines the cartesian product of classes \cprod{\mathcal{A}}{\mathcal{B}}:

      \[  \cprod{\mathcal{A}}{\mathcal{B}}= \{\langle x,y\rangle\ |\ x\in_1\mathcal{A}\ \&\ y\in_1 \mathcal{B}\ \}\]
And a set theoretic abstract defines the collection  of all functions \arrow{\mathcal{F}}{\mathcal{A}}{\mathcal{B}}:

\[  \mathcal{B^A}= \{\mathcal{F}\ |\ 
   \mathcal{F}\subseteq_1 \cprod{\mathcal{A}}{\mathcal{B}}\   \&\  (\forall x\in_1\mathcal{A})(\exists! y\in_1\mathcal{B})\ \langle x,y\rangle \in_1 \mathcal{F}   \}\]

Another example is the abstract for the class of all small categories:
 \[\{\langle C_0,C_1,d_0,d_1,m\rangle|\ Cat( C_0,C_1,,d_0,d_1,m)\}\]
 Here $Cat$ is a formula saying $d_0,d_1$ are functions between the sets \clarrow{C_1}{C_0} and $m$ is a partially defined function \clarrow{\cprod{C_1}{C_1}}{C_1} fulfilling the category axioms.

A 5-tuple of sets $\langle C_0,C_1,d_0,d_1,m\rangle$ is a set.  But we also want a collection of all class-sized categories while a 5-tuple of classes is not naturally a class.  So we take $n$-tuples of classes as primitive.  There is an abstract  
\[\{\langle \mathcal{C}_0,\mathcal{C}_1,\mathcal{D}_0,\mathcal{D}_1, \mathcal{M}\rangle|\ Cat(\mathcal{C}_0,\mathcal{C}_1,\mathcal{D}_0,\mathcal{D}_1, \mathcal{M}))\}\]
saying the classes  $\mathcal{C}_0,\mathcal{C}_1,\mathcal{D}_0,\mathcal{D}_1, \mathcal{M}$ fulfill the category axioms.  It indicates the 5-tuple collection of all class categories, and it is set theoretic since it only quantifies over elements of the classes involved.

We adapt rules from  \citet[p.~77--80]{TakeutiConserv}.  Our basic types are $Set,\ Class$ (of sets), and $Collection$ (of classes).  For any types $\tau_1,\tau_2$ there is a product type \cprod{\tau_1}{\tau_2}.  Abstracts are as defined here:

\begin{enumerate}

\item[6.] If $\Psi(v_1,\dots,v_n)$ is a set theoretic formula with variables $v_1,\dots,v_n$  of types $\tau_1,\dots, \tau_n$ then 
   \[\{\langle v_1,\dots,v_n\rangle|\Psi(v_1,\dots,v_n)\}\]
is an abstract of type \cprod{\tau_1}{\cprod{\dots}{\tau_n}}.  The indicated variables need not actually occur in $\Psi$, and other free variables of any type may occur.

\item[$6'$.]  Given abstracts $A_1,\dots,A_n$ of types $\tau_1,\dots, \tau_n$ respectively, and abstract $\{\langle v_1,\dots,v_n\rangle| \Psi(v_1,\dots,v_n)\}$  with variables of the same types, the expression $\langle A_1,\dots,A_n \rangle \in\{\langle v_1,\dots,v_n\rangle| \Psi(v_1,\dots,v_n)\}$ is equivalent to $\Psi( A_1,\dots,A_n)$.

\item[7.] If $\alpha$ is a free variable of type  \cprod{\tau_1}{\cprod{\dots}{\tau_n}} and $A_1,\dots,A_n$ are abstracts of types $\tau_1,\dots, \tau_n$ then $\langle A_1,\dots,A_n \rangle \in \alpha$ is a formula.
\end{enumerate}

As quantifier rules:
   \[ \forall \alpha \Psi(\alpha) \text{ implies } \Psi(A)\]
for any formula $\Psi(\alpha)$ and abstract $A$ of the same type as variable $\alpha$.  And given any proof of $\Psi(\alpha)$, with variable $\alpha$ not in any assumption, conclude $\forall \alpha\Psi(\alpha)$.  Define $\exists \alpha \Psi(\alpha)$ as $\neg \forall \alpha \neg \Psi(\alpha)$.  This gives class and collection comprehension:  For each set theoretic formula $\Psi(v)$ with set or class variable $v$, the equivalence
  \[\Psi(v) \leftrightarrow v\in_i \{v|\Psi(v)\} \text{\quad implies \quad } 
	\exists \psi \forall v (\Psi(v) \leftrightarrow v\in_i \psi) \quad \text{for\ } i=1,2\]

The identity axiom connects classes to sets:  
   \[ \forall \mathcal{A} \forall x \forall y\, (\, ( x=y\ \&\ x\in_1 \mathcal{A})
						 \rightarrow y\in_1 \mathcal{A}\, )\]
A class $\mathcal{A}$ might be a set in the sense of having the same elements as some set
  \[  \exists A \forall x (x\in A \leftrightarrow x\in_1 \mathcal{A})\]
We express this set theoretic formula informally by saying $\mathcal{A}$ is small or is a set.  The set $A$ is uniquely determined and we can work with $\mathcal{A}$ by working with $A$,

We use no identity relation for classes or collections.  This follows \citet{TakeutiConserv} where the absence of higher-type identity facilitates the conservative extension proof.   And it in practice large-structure categories like toposes are generally compared in terms of definable equivalence rather than identity. 
 
\section{Category theory in MTT}\label{S:CatsMTT}
A class category is a 5-tuple of classes $\langle \mathcal{C}_0,\mathcal{C}_1,\mathcal{D}_0,\mathcal{D}_1,\mathcal{M}\rangle$ satisfying the axioms.  Elements of $\mathcal{C}_0$ and $\mathcal{C}_1$ are sets so the axioms are set theoretic and there is a collection category  $\mathfrak{CAT}$ of all class categories.  The abstract in Section~\ref{S:MTT} gives the collection  $\mathfrak{CAT}_0$ of all class categories.  Similar ones work for all functors, and so on.

Take the class $\mathcal{U}$ of all sets as \emph{universe}.  With the class of all functions between sets it provides a class category \smcat{SET}.  This $\mathcal{SET}$ is a $\mathcal{U}$-category by this definition:.

\begin{definition}A $\mathcal{U}$-category, also called a \emph{locally small category} is a category with a class of objects and and a class of arrows such that every set of objects has a set of all arrows between them.\footnote{\citet[p.~5]{GroPrefaisceaux} reject this definition because presheaf categories should be $\mathcal{U}$-categories while their definition at the time made presheaves too big to be in $\mathcal{U}$.  Our Section~\ref{SS:Presheaves} uses the later \emph{Grothendieck construction} so presheaves are indexed sets.} 
\end{definition}

There is a  $\mathcal{U}$-category $\mathcal{CAT}$ of all small categories using the abstract in Section~\ref{S:MTT} for the class $\mathcal{CAT}_0$ of all small categories. Analoguous abstracts give the class $\mathcal{CAT}_1$  of all small functors and the class graphs of the domain, codomain, and composition functions.  Section~\ref{SS:Smcats} shows it is locally small.

\subsection{Sheaf and presheaf toposes}

Section~\ref {SS:Presheaves} proved the category of presheaves on a small category \smcat{C} is locally small.  Call that category of presheaves \smcat{\widehat{C}}.  It is indicated by a 5-tuple of classes:
   \[ \{\langle \mathcal{F_C}_0, \mathcal{F_C}_1,\mathcal{D}_0,\mathcal{D}_1,   	
	\mathcal{M}\rangle|\  \begin{cases}
                F\in_1  \mathcal{F_C}_0 \text{ iff $F$ is a presheaf on \smcat{C}} \\  
                \eta \in_1  \mathcal{F_C}_1 \text{ iff $\eta$ is a presheaf transform on \smcat{C}} \\
                  \text{ etc.}
                       \end{cases}\}\]
All these formulas are set theoretic.  Here  $\smcat{C}$ abbreviates a 5-tuple $\langle C_0,C_1,d_0,d_1,m\rangle$ of free variables of set type and conditions saying they form a small category, so the abstract indicates a variable presheaf category \smcat{\widehat{C}} depending on \smcat{C}.  We can also abstract over all these variables at once to form 
 \begin{multline*}\label{M:Yoneda} \{\langle  C_0,C_1,d_0,d_1,m,\mathcal{F_C}_0, \mathcal{F_C}_1,
                 \mathcal{D}_0,\mathcal{D}_1,   	
	\mathcal{M}\ \rangle\ |\\ \langle\mathcal{F_C}_0, \mathcal{F_C}_1,\mathcal{D}_0,
      \mathcal{D}_1, 	\mathcal{M}\rangle \text{ is the presheaf category on\ }
                  \langle  C_0,C_1,d_0,d_1,m\rangle  \}\end{multline*}
indicating the class of all pairs of a small category \smcat{C} and its presheaf category  \smcat{\widehat{C}}.

In MTT, for each small category \smcat{C}, the Yoneda operation $R_{(\_)}$ is an actual functor \arrow{R_{(\_)}}{\smcat{C}}{\smcat{\widehat{C}}} called the \emph{Yoneda embedding}.   Compare Section~\ref{SS:Yoneda}.

For any small site  $\langle \smcat{C},J\rangle$ MTT provides a category of sheaves called $\smcat{\widetilde{C}}_J$.   As a full subcategory of a presheaf category it is locally small.  The definition of the associated sheaf \arrow{i}{F}{\mathbf{a}F} in Section~\ref{SS:topologies} says sheafification \arrow{\mathbf{a}}{\smcat{\widehat{C}}}{\smcat{\widetilde{C}}_J} is left adjoint to the inclusion \clmarrow{\smcat{\widetilde{C}}_J}{\smcat{\widehat{C}}}.  Proofs in SGA~4~II and \citep[pp.~227ff.]{MacMoer}~work in MC and show sheafification preserves finite limits. 

A Grothendieck topos in MTT is any class category equivalent to $\smcat{\widetilde{C}}_J$ for some small site.   It is locally small since equivalence preserves the size of arrow sets.   MTT does not prove there is a collection of all Grothendieck toposes, since the definition of a Grothendieck topos quantifies over classes.  But MTT does prove there is a collection of all sheaf categories on small sites which we treat in Section~\ref{S:Morphism}.

\begin{theorem}\label{T:GRTopos}Every Grothendieck toposes in MTT is a model of elementary topos theory.  See for example \citep{JTop}.
\end{theorem}
\begin{proof} This is Theorem~\ref{T:Topos} plus the fact that each Grothendieck topos forms a class in MTT and that the axioms of elementary topos theory are isomorphism invariant.  
\end{proof}

\subsection{Cohomology in MTT}\label{S:CohomologyMTT}
A sheaf of modules over a sheaf of rings on any small site $\langle \smcat{C},J\rangle$ is just a module $M$ on a ring $R$ in the sheaf topos $\smcat{\widetilde{C}}_J$.  All commutative algebra that does not use excluded middle or the axiom of choice holds in every Grothendieck topos by Thm.~\ref{T:Topos}.     

For any ring $R$ in any sheaf topos $\smcat{\widetilde{C}}_J$,   MTT gives a  $\mathcal{U}$-category $\mathcal{MOD}_R$ of all $R$-modules.   The usual constructions of biproducts, kernels, and cokernels are bounded so they work in MC, so they show in MTT that  $\mathcal{MOD}_R$ is an Abelian category. 

Section~\ref{SS:injectives} defined cohomology groups $H^n(E,M)$ in MC.  In MTT we define cohomology functors \arrow{H^n}{\mathcal{MOD}_R}{\mathcal{AB}} from sheaves of modules to ordinary Abelian groups. The construction in  Section~\ref{SS:injectives} was explicit (not using choice) and set theoretic so MTT can express it by a class abstract.  

 MTT can give the usual definition of a universal $\delta$-functor~\citep[p.~204]{HartAG}.  Every  left exact functor \arrow{F}{\mathcal{MOD}_R}{\mathcal{AB}} has \emph{right derived functors}
  \[ F\cong R^0 F,\ R^1 F,\ \dots,\ R^n F,\ \dots\]
defined up to isomorphism either as a  universal $\delta$-functor over $F$, or as an effaceable $\delta$-functor over $F$.  See~\citep[p.~141]{Tohoku}.

The cohomology functors $H^i,i\leq n$ are derived functors of the global section functor  \arrow{\Gamma}{\mathcal{MOD}_R}{\mathcal{AB}} which takes each module to its group of global sections.

\section{Large-structure tools}\label{S:largestructure}

\subsection{Geometric morphisms}\label{S:Morphism}
A \emph{geometric morphism} of toposes is an adjoint pair of functors \adj{f^*}{f_*}{\smcat{E}}{\smcat{E}'} where the left adjoint \arrow{f^*}{\smcat{E}'}{\smcat{E}} is also left exact.   Then \arrow{f_*}{\smcat{E}}{\smcat{E}'} is called the \emph{direct image} functor, and $f^*$ the \emph{inverse image} functor.  The standard, published theory of geometric morphisms among Grothendieck toposes largely applies.  The objects and arrows of any Grothendieck topos are sets, and most standard constructions are all bounded.  For example, each Grothendieck topos \smcat{E} has a geometric morphism  \adj{\Delta}{\Gamma}{\smcat{E}}{\smcat{SET}} with the \emph{global section functor} $\Gamma$ taking each object $A\in \smcat{E}$ to the set of arrows \clarrow{1}{A}.  The usual argument shows this is up to equivalence the only geometric morphism from \smcat{E} to \smcat{SET}.

For other examples, MC proves any continuous function \arrow{f}{X}{X'} between topological spaces induces suitable operations on sheaves and their transforms on those spaces.  So MTT proves $f$ induces a geometric morphism \arrow{f^*,f_*}{Top(X)}{Top(X')} between the sheaf toposes, and given suitable separation conditions on the spaces every geometric morphism arises from a unique continuous function. See \citet[p.~348]{MacMoer}.  

Grothendieck toposes are defined in MTT but the definition quantifies over functors of class type, saying a category is a Grothendieck topos if there exists a functor equivalence beween it and some sheaf topos.  So MTT cannot prove there is a collection of all Grothendieck toposes.   It can prove there is a collection $\mathfrak{Top}_0$ of all sheaf toposes, and thus all Grothendieck toposes up to equivalence.   This has abstract
    \[\{ \langle \mathcal{S}_0,\mathcal{S}_1,d_0,d_1,m\rangle|\ \exists\text{ a small site }
               \langle \smcat{C},J\rangle\ 
        \begin{cases} X\in_1 \mathcal{S}_0 \leftrightarrow X \text{ is a sheaf on } \langle 
                                                     \smcat{C},J\rangle  \\
                 f\in_1 \mathcal{S}_1 \leftrightarrow f \text{ is a sheaf transform}  \\
                   etc.  
         \end{cases} \}\]
Similar abstracts give a collection  $\mathfrak{Top}_1$  of all geometric morphisms between sheaf toposes, and a collection of all natural transformations between these morphisms.  These form a 2-category $\mathfrak{Top}$ of Grothendieck toposes.  Cf.~\cite[p.~26]{JTop}. 

The standard theorems on $\mathfrak{Top}$ follow in MTT.  They make elementary use of classes, and quantify only over sheaves, transforms, and other sets.

\subsection{Sites}\label{S:Utopos}
A presheaf on a $\mathcal{U}$-category $\smcat{C}$ is a $\smcat{C}_0$ indexed class \arrow{\gamma_0}{F_0}{C_0} with action $e_{\mathcal{F}}$ analogously to Section~\ref{SS:Presheaves}.   A $\mathcal{U}$-presheaf or \emph{locally small presheaf} on \smcat{C} is a presheaf whose values are all sets, that is such that the restriction to any small subcategory $\smcat{C'}\subseteq \smcat{C}$ is small.

A $\mathcal{U}$-site, or \emph{locally small site}  $\langle \smcat{C},J\rangle$, is a site with locally small \smcat{C}.  A $\mathcal{U}$-sheaf is a locally small presheaf with the sheaf property.  Local smallness only quantifies over sets: every set of objects in a class category has a set of values.  So MTT can invoke local smallness in abstracts.  Thus every $\mathcal{U}$-site  $\langle \smcat{C},J\rangle$ has a class category  $\smcat{\widetilde{C}}_{J}$  of all  $\mathcal{U}$-sheaves.  A \emph{class topos} is any class category equivalent to  $\smcat{\widetilde{C}}_{J}$ for some $\mathcal{U}$-site.  For suitably bounded $\mathcal{U}$-sites, these are Grothendieck toposes:

\begin{theorem}[Comparison lemma]\label{Th:Compare} Let $\mathcal{U}$-site  $\langle \smcat{C}',J'\rangle$ have a full and faithful functor \arrow{u}{\smcat{C}}{\smcat{C}'} from a small category \smcat{C} where every object of \smcat{C'} has at least one $J'$-cover by objects $u(A)$ for objects $A$ of \smcat{C}.  Then $J'$ induces a topology $J$ on \smcat{C} making $\smcat{\widetilde{C}}_{J}$ and $\smcat{\widetilde{C}'}_{J'}$ equivalent categories.
\end{theorem}
\begin{proof}This is case i)$\Rightarrow$ii) of SGA~4 III.4.1 (p.~288).  Verdier's small categories are sets  for us, as are his functors $u_!,u^*,u_*$.  The constructions are bounded.  The proof by \citet[p.~588]{MacMoer} also adapts to MTT.
\end{proof}

\begin{corollary}\label{Cor:topos}Any $\mathcal{U}$-category \smcat{E} with a set of generators $\{G_i|i\in I\}\in \mathcal{U}$ and with every $\mathcal{U}$-sheaf for the canonical topology representable, is a Grothendieck topos.
\end{corollary}
\begin{proof}See the canonical topology in any topos theory text.  The representability assumption says \smcat{E} is equivalent to the category of canonical $\mathcal{U}$-sheaves.  Apply the theorem to \smcat{C'=E} and \smcat{C} the full subcategory of objects in $G$.
\end{proof}

\begin{theorem}\label{Th:properties} For any small site $\langle \smcat{C},J\rangle$ the sheaf topos $\smcat{\widetilde{C}}_{J}$ has:
\begin{itemize}
\item [a)] a limit for every finite diagram.
\item [b)] a coproduct for each set of sheaves, and these are stable disjoint unions.
\item [c)] a stable quotient for every equivalence relation.
\item [d)] a set $\{G_i|i\in I\}$ of generators.
\end{itemize}
\end{theorem}
\begin{proof}   Section~\ref{SS:Presheaves} proved most of this for presheaf categories.  The sheaf case follows from sheafification described in Section~\ref{SS:topologies}.  See SGA~4 II.4 (p.~235) and SGA~4 IV.1.1.2 (p.~302); or see \citet[pp.~24ff.]{MacMoer}.
\end{proof}

In fact $\smcat{\widetilde{C}}_{J}$ has limits for every small diagram, but Theorem~\ref{Th:Giraud} below refers to this list as given.  The list amounts to saying \smcat{\widetilde{C}} is an elementary topos with small coproducts and a small generator~\citep[pp.~591]{MacMoer}.

\begin{theorem}[Giraud theorem]\label{Th:Giraud} Any $\mathcal{U}$-category \smcat{E} with the properties listed in Theorem~\ref{Th:properties} is a Grothendieck topos. 
\end{theorem}
\begin{proof}  The proof by \citet[pp.~578ff.]{MacMoer}\@ is easily cast in MTT.  As they do, define \smcat{C} to be the full subcategory of \smcat{E} on the set of generators.  It is small since \smcat{E} is locally small.  Take their functors \adj{(\mathrm{Hom}_{\smcat{E}})}{(\_ \otimes_{\smcat{C}}A)}{\smcat{SET}^{\smcat{C}^{op}}}{\smcat{E}} as class functors between class categories.
\end{proof}
\begin{corollary}Every Grothendieck topos is equivalent to some sheaf topos  on a subcanonical site with all finite limits.
\end{corollary}
\begin{proof}  After \citet[pp.~578ff.]{MacMoer}, it remains to prove in MTT that every small category \smcat{C} has a small full subcategory $\smcat{C'}\subseteq \widehat{\smcat{C}}$ of presheaves containing the representables and closed under finite limits.  Since $\widehat{\smcat{C}}$ is locally small it suffices to find a set of presheaves including the representables and closed under finite limits.  Limits of presheaves are computed pointwise \citep[p.~116]{CfWM}, and a product of equalizers is an equalizer.  So we must show for each set of sets there is a set of all finite products of those sets, which follows if we know for each single set $A$ there is a set of all finite powers $A^n$.  To prove that, code an $n$-tuple of elements of $A$ as a partial function \clarrow{\mathbb{N}}{A} defined for $0\leq i < n$. 
\end{proof}

Where SGA~4 invokes two universes $U\in V$ the larger is always just a shorthand for dealing with definable subclasses of $U$ as we can do in MTT\@.  See e.g.~the Giraud theorem (IV.1.2) and sheaf multilinear algebra (IV.10).

\subsection{Duality and derived categories}\label{SS:derivedcats} 
\begin{quotation}The chief ideas of [Grothendieck duality] were known to me since 1959, but the lack of adequate foundations for homological algebra  prevented me attempting a comprehensive  revision.  This gap in foundations is about to be filled by Verdier's dissertation, making a satisfactory presentation possible in principle. (Grothendieck quoted by \citealp[p.~III]{HartResidues})
\end{quotation}

\citet{GroDuality} finds his duality theorem too limited.  It was essentially as in \citet{AltmanKleiman}: certain cohomology groups (and related groups) of nonsingular projective schemes are isomorphic in a natural way.  The proof invokes proper class categories but really only quantifies over sheaves and modules.  It can be given in MC.  \citet[p.~486]{WilesFerm} calls it ``explicit duality over fields.'' 

\citet[pp.~112--15]{GroCoh} explains why duality should reach farther.  By 1959 he believed the most unified and general tool is \emph{derived categories}, now standard for Grothendieck duality.  ``Miraculously, the same formalism applies in \'etale cohomology, with quite different proofs''~\cite[p.~17]{DelQuel}.  Deligne uses them for \'etale Poincar\'e duality in SGA~4~XVII,~XVIII and \citep{DelDepart}.

Cohomology takes a module $M$ on a scheme $X$ and deletes nearly all its structure, highlighting just a little of it in the groups $H^n(X,M)$.  The \emph{derived category} $D(X)$ of modules on $X$ deletes much of the same information but not all.  Some manipulations work at this level which are obscured by excess detail at the level of modules and are impossible for lack of detail at the level of  cohomology.

A scheme map \arrow{f}{X}{Y} sets up complicated relations between cohomology over $X$ and $Y$.  The successive effect on cohomology of $f$ and a further \arrow{g}{Y}{Z} is not fully determined by the separate effects of $f$ and $g$ (those determine it only up to a spectral sequence).  A functor \arrow{Rf_*}{D(X)}{D(Y)} between derived categories approximates the effect of $f$ on cohomology so that the approximation of successive effects is precisely the composite of the approximations:   
  \[ \xymatrix@R=1pc{ & Y \ar[dr]^g  &&& D(Y) \ar[dr]^{Rg_*} \\ X \ar[ur]^f \ar[rr]_{gf} && Z 
	& D(X) \ar[ur]^{Rf_*}    \ar[rr]_{R(gf)_* \cong \ Rg_*Rf_*} && DZ) }\]
 
All variants of Grothendieck duality being developed today say the functor $Rf_*$ has a right adjoint $Rf^!$, with further properties under some conditions on $f$.  The adjunction contains very much information.

The set theoretic issue is to form certain \emph{categories of fractions}. In any small or class category \smcat{C} each suitable class $\Sigma$ of arrows has a category of fractions $\smcat{C}[\Sigma^{-1}]$ inverting each arrow in $\Sigma$.  It has the same objects, while an arrow \clarrow{A}{B} in $\smcat{C}[\Sigma^{-1}]$ is represented by a pair of arrows in \smcat{C}:
   \[\xymatrix{ A  & C \ar[l]_s \ar[rr]^f && B &  s\in\Sigma}\]  
We define an equivalence relation on these pairs, and a composition rule so a pair $\langle s,f\rangle$ acts like a composite \arrow{fs^{-1}}{A}{B} even if $s$ has no inverse in \smcat{C}. 

The derived category $D(X)$ starts with the category $\smcat{K}(X)$ whose objects are complexes of quasi-coherent sheaves of modules over a scheme $X$
   \[ \xymatrix{ \cdots \ar[r] & M_{i-1} \ar[r] &M_i \ar[r] &M_{i+1} \ar[r] & M_{i+2} \ar[r] & \cdots}\]
and arrows are \emph{homotopy classes} of maps between complexes.  Quasi-coherent sheaves are those closest to the geometry of a scheme, but this sets no bound on cardinality and does not affect the set theory involved.  Complexes and homotopy classes are sets, provably existing in MC.\@ The derived category $D(X)$ is a certain calculus of fractions on $\mathcal{K}(X)$ \citep[pp.~678ff.]{EisComm}.  

\citet[p.~386]{WeibelIntro} cuts the classes of fractions down to sets for many important cases including modules on schemes.  But he uses countable replacement so that sequences of cardinals have suprema.  MTT avoids replacement and these sequences of cardinals and does not limit the cases.

Here the class category is $\smcat{K}(X)$ and $\Sigma$ is the class of \emph{quasi-isomorphisms}, the homotopy classes inducing isomorphisms in all degrees of cohomology.   For fixed $A,B$ the relevant pairs are
   \[\xymatrix{ A  & C \ar[l]_s \ar[rr]^f && B &  s\text{ any quasi-isomorphism}}\]  
Each single equivalence class in $\smcat{C}[\Sigma^{-1}]$ involves a proper class of pairs with different $C$.  The collection $D(X)_1$ of arrows of $D(X)$  is the collection of these equivalence classes, while the class of objects is the class $D(X)_0=\smcat{K}(X)_0$ of complexes.

The key point conceptually and for MTT is that the definition of $D(X)_1$ depends on (infinitely many) complexes of modules making (infinitely many) finite diagrams commute.  It is expressed by a set theoretic abstract.   The graphs of domain, codomain, and composition are similar.  MTT proves there is a derived category $D(X)$, with a class of objects and collection of arrows.  

So current work on Grothendieck duality is formalizable in MTT.  For debate over mathematical strategies (not foundations) see \citet[preface]{ConradDuality}, Lipman in \citep[pp.~7--9]{LipmanHashimoto}, and \citet[pp.~294--300]{NeemanDerive}.  \citet[pp.~1--13]{HartResidues} describes an ``ideal form'' of the theorem and suggests ``Perhaps some day this type of construction will be done more elegantly using the language of fibred categories and results of Giraud's thesis''~(p.~16).

\subsection{Fibred categories}
Universes first appeared in print in SGA~1~VI on fibred categories. They are a way to treat a class or category of categories as a single category.  So SGA~4~VI calculates limits of families of Grothendieck toposes by using fibred toposes.  In much of SGA~4 fibred toposes are presented by fibred sites.  The logical issues are essentially the same as in Section~\ref{S:Utopos}.  Many applications can be cast in MC in terms of sites, while the general facts are clearer and more concise in MTT using toposes and fibred families of them.   The latter requires no stronger logical foundation than the former.

\section{A proof of Fermat's Last Theorem in PA?}\label{S:Fermat} 
We have founded the whole SGA for arbitrary sites, while individual proofs in number theory use only low degree cohomology of sites close to arithmetic.  Detailed bounds may suffice to get existing proofs into $n$-order arithmetic for relatively low $n$, as in Section~\ref{S:Third}.  That might be a good context for such hard logical analysis as \citet{MacintyreImpact} begins for FLT\@.  More work might bound the constructions within a conservative extension of PA~\citep{TakeutiConserv} to show some existing proof of FLT works essentially in PA.  It might help further reduce the proof to Exponential Function Arithmetic (EFA) as conjectured in \citep{FriedmanConcrete}.  Such estimates are likely to be difficult.  This is no logical end run around serious arithmetic.

Not motivated by concern with logic, \citet{KisinModuli} extends and simplifies~\citep{WilesFerm}, generally using geometry less than commutative algebra, visibly reducing the demands on set theory.  And \citet{KisinBarsotti} completes a different proof of FLT by a strategy of Serre advanced by Khare and Wintenberger.

\section*{Acknowledgments}
It is a pleasure to thank people who contributed ideas to this work, which does not mean any of them shares any given viewpoint here.  I thank especially Jeremy Avigad, Steve Awodey, John Baldwin, Brian Conrad, Walter Dean, Pierre Deligne, Fran\c cois  Dorais, Adam Epstein, Thomas Forster, Harvey Friedman, Sy David Friedman, Steve Gubkin, Michael Harris, Wiliam Lawvere, Angus Macintyre, Barry Mazur, Michael Rathjen, Michael Shulman, Jean-Pierre Serre, and Robert Solovay.  
\vspace{5ex}

\bibliographystyle{apalike}
%\bibliography{cohomologyrefs,MacLane}

\begin{thebibliography}{}

\bibitem[Altman and Kleiman, 1970]{AltmanKleiman}
Altman, A. and Kleiman, S. (1970).
\newblock {\em An Introduction to {G}rothendieck Duality Theory}.
\newblock Springer Lecture Notes in Mathematics no. 146. Springer-Verlag, New
  York.

\bibitem[Artin et~al., 1964]{SGA4mimeo}
Artin, M., Grothen\-dieck, A., and Verdier, J.-L. (1964).
\newblock {\em Cohomologie Etale des Sch{\'e}mas}, volume~1 of {\em
  S{\'e}minaire de g{\'e}om{\'e}trie alg{\'e}brique 1963--64}.
\newblock Institut des Hautes \'Etudes Scientifiques.
\newblock Mimeographed reports of the seminar later published in greatly
  expanded form as SGA~4.

\bibitem[Artin et~al., 1972]{SGA4}
Artin, M., Grothen\-dieck, A., and Verdier, J.-L. (1972).
\newblock {\em Th{\'e}orie des Topos et Cohomologie Etale des Sch{\'e}mas}.
\newblock S{\'e}minaire de g{\'e}om{\'e}trie alg{\'e}brique du Bois-Marie, 4.
  Springer-Verlag.
\newblock Three volumes, cited as SGA~4.

\bibitem[Baer, 1940]{BaerAb}
Baer, R. (1940).
\newblock Abelian groups that are direct summands of every containing abelian
  group.
\newblock {\em Bulletin of the American Mathematical Society}, 46:800--06.

\bibitem[Barr, 1974]{BarrPoints}
Barr, M. (1974).
\newblock Toposes without points.
\newblock {\em Journal of Pure and Applied Algebra}, 5:265--80.

\bibitem[Blass, 1979]{BlassInj}
Blass, A. (1979).
\newblock Injectivity, projectivity, and the axiom of choice.
\newblock {\em Transactions of the American Mathematical Society}, 255:31--59.

\bibitem[Conrad, 2000]{ConradDuality}
Conrad, B. (2000).
\newblock {\em Grothendieck duality and base charge}.
\newblock Number 1750 in Lecture Notes in Mathematics. Springer-Verlag, New
  York.

\bibitem[Deligne, 1977]{DelDepart}
Deligne, P. (1977).
\newblock Cohomologie \'{e}tale: les points de d{\'e}part.
\newblock In Deligne, P., editor, {\em Cohomologie \'{E}tale}, pages 4--75.
  Springer-Verlag.

\bibitem[Deligne, 1998]{DelQuel}
Deligne, P. (1998).
\newblock Quelques id\'{e}es ma\^{i}tresses de l'\oe uvre de {A.
  Grothen\-dieck}.
\newblock In {\em Mat\'{e}riaux pour l'Histoire des Math\'{e}matiques au XX$\sp
  {\rm e}$ Si\`{e}cle (Nice, 1996)}, pages 11--19. Soc. Math. France.

\bibitem[Eckmann and Schopf, 1953]{EckmannSchopf}
Eckmann, B. and Schopf, A. (1953).
\newblock {\"U}ber injektive {M}oduln.
\newblock {\em Archiv der Mathematik}, 4(2):75--78.

\bibitem[Eisenbud, 1995]{EisComm}
Eisenbud, D. (1995).
\newblock {\em Commutative Algebra}.
\newblock Springer-Verlag, New York.

\bibitem[Friedman, 2010]{FriedmanConcrete}
Friedman, H. (2010).
\newblock Concrete mathematical incompleteness.
\newblock On-line at www.math.ohio-state.edu/~friedman/.
\newblock Lecture at University of Cambridge, UK.

\bibitem[Grothen\-dieck, 1957a]{Tohoku}
Grothen\-dieck, A. (1957a).
\newblock Sur quelques points d'alg{\`e}bre homologique.
\newblock {\em T\^{o}hoku Mathematical Journal}, 9:119--221.

\bibitem[Grothen\-dieck, 1957b]{GroDuality}
Grothen\-dieck, A. (1957b).
\newblock Th{\'e}or{\`e}mes de dualit{\'e} pour les faisceaux alg{\'e}briques
  coh{\'e}rents, expos\'{e} 149.
\newblock In {\em S\'{e}minaire Bourbaki}. Secr{\'e}tariat math{\'e}matique,
  Universit{\'e} Paris, Paris.

\bibitem[Grothen\-dieck, 1958]{GroCoh}
Grothen\-dieck, A. (1958).
\newblock The cohomology theory of abstract algebraic varieties.
\newblock In {\em Proceedings of the International Congress of Mathematicians,
  1958}, pages 103--18. Cambridge University Press.

\bibitem[Grothen\-dieck, 1971]{SGA1}
Grothen\-dieck, A. (1971).
\newblock {\em Rev\^etements {\'E}tales et Groupe Fondamental}.
\newblock S{\'e}minaire de g{\'e}om{\'e}trie alg{\'e}brique du Bois-Marie, 1.
  Springer-Verlag.
\newblock Cited as SGA 1.

\bibitem[Grothen\-dieck,   87]{ReS}
Grothen\-dieck, A. (1985--87).
\newblock {\em R{\'e}coltes et Semailles}.
\newblock Universit{\'e} des Sciences et Techniques du Languedoc, Montpellier.
\newblock Published in several successive volumes.

\bibitem[Grothen\-dieck and Dieudonn{\'e}, 1964]{EGAIV.1}
Grothen\-dieck, A. and Dieudonn{\'e}, J. (1964).
\newblock {\em {\'E}l{\'e}ments de G{\'e}om{\'e}trie Alg{\'e}brique IV:
  \'{E}tude locale des sch\'{e}mas et des morphismes de sch\'{e}mas,
  Premi\`{e}re partie}.
\newblock Number~20 in Publications Math{\'e}matiques. Institut des Hautes
  {\'E}tudes Scientifiques, Paris.

\bibitem[Grothen\-dieck and Dieudonn{\'e}, 1971]{EGAIspringer}
Grothen\-dieck, A. and Dieudonn{\'e}, J. (1971).
\newblock {\em {\'E}l{\'e}ments de G{\'e}om{\'e}trie Alg{\'e}brique I}.
\newblock Springer-Verlag.

\bibitem[Grothendieck and Verdier, 1972]{GroPrefaisceaux}
Grothendieck, A. and Verdier, J.-L. (1972).
\newblock Pr{\'e}faisceaux.
\newblock In Artin, M., Grothen\-dieck, A., and Verdier, J.-L., editors, {\em
  Th{\'e}orie des Topos et Cohomologie Etale des Sch{\'e}mas}, volume~1 of {\em
  S{\'e}minaire de g{\'e}om{\'e}trie alg{\'e}brique du Bois-Marie, 4}, pages
  1--218. Springer-Verlag.

\bibitem[Hartshorne, 1966]{HartResidues}
Hartshorne, R. (1966).
\newblock {\em Residues and Duality, Lecture Notes of a Seminar on the Work of
  {A.~G}rothendieck given at {H}arvard 1963--64}.
\newblock Number~20 in Lecture Notes in Mathematics. Springer-Verlag, New York.

\bibitem[Hartshorne, 1977]{HartAG}
Hartshorne, R. (1977).
\newblock {\em Algebraic Geometry}.
\newblock Springer-Verlag.

\bibitem[Johnstone, 1977]{JTop}
Johnstone, P. (1977).
\newblock {\em Topos Theory}.
\newblock Academic Press.

\bibitem[Karazeris, 2004]{Karazeris}
Karazeris, P. (2004).
\newblock Notions of flatness relative to a {G}rothendieck topology.
\newblock {\em Theory and Applications of Categories}, 12:225--36.
\bibitem[Kisin, 2009a]{KisinBarsotti}
Kisin, M. (2009a).
\newblock Modularity of 2-adic {B}arsotti-{T}ate representations.
\newblock {\em Inventiones Mathematicae}, 178(3):587--634.

\bibitem[Kisin, 2009b]{KisinModuli}
Kisin, M. (2009b).
\newblock Moduli of finite flat group schemes, and modularity.
\newblock {\em Annals of Mathematics}, 170(3):1085--1180.

\bibitem[Lawvere, 1965]{LawElem}
Lawvere, F.~W. (1965).
\newblock An elementary theory of the category of sets.
\newblock Lecture notes of the Department of Mathematics, University of
  Chicago.
\newblock Reprint with commentary by the author and Colin McLarty in:
  \emph{Reprints in Theory and Applications of Categories}, No. 11 (2005) pp.
  1--35, on-line at http://138.73.27.39/tac/reprints/articles/11/tr11abs.html.

\bibitem[Lipman and Hashimoto, 2009]{LipmanHashimoto}
Lipman, J. and Hashimoto, M. (2009).
\newblock {\em Foundations of {G}rothendieck Duality for Diagrams of Schemes}.
\newblock Springer-Verlag.

\bibitem[Mac~Lane, 1998]{CfWM}
Mac~Lane, S. (1998).
\newblock {\em Categories for the Working Mathematician}.
\newblock Springer-Verlag, New York, 2nd edition.


\bibitem[Mac~Lane and Moerdijk, 1992]{MacMoer}
Mac~Lane, S. and Moerdijk, I. (1992).
\newblock {\em Sheaves in Geometry and Logic}.
\newblock Springer-Verlag.

\bibitem[Macintyre, 2011]{MacintyreImpact}
Macintyre, A. (2011).
\newblock The impact of {G\"o}del's incompleteness theorems on mathematics.
\newblock In {\em {K}urt {G\"o}del and the Foundations of Mathematics: Horizons
  of Truth}, pages 3--25.
\newblock Proceedings of G{\"o}del Centenary, Vienna, 2006.

\bibitem[Maranda, 1964]{MarandaInj}
Maranda, J.-M. (1964).
\newblock Injective structures.
\newblock {\em Transactions of the American Mathematical Society}, 110:98--135.

\bibitem[Mathias, 2001]{MathThe}
Mathias, A. R.~D. (2001).
\newblock The strength of {M}ac {L}ane set theory.
\newblock {\em Annals of Pure and Applied Logic}, 110:107--234.

\bibitem[Neeman, 2010]{NeemanDerive}
Neeman, A. (2010).
\newblock Derived categories and {G}rothendieck duality.
\newblock In Holm, T., J{\o}rgensen, P., and Rouquier, R., editors, {\em
  Triangulated categories}, pages 290--350. Cambridge University Press.

\bibitem[Takeuti, 1978]{TakeutiConserv}
Takeuti, G. (1978).
\newblock A conservative extension of {P}eano {A}rithmetic.
\newblock In {\em Two Applications of Logic to Mathematics}, pages 77--135.
  Princeton University Press.

\bibitem[Takeuti, 1987]{TakeutiProof}
Takeuti, G. (1987).
\newblock {\em Proof Theory}.
\newblock Elsevier Science Ltd, 2nd edition.

\bibitem[Tamme, 1994]{TamEt}
Tamme, G. (1994).
\newblock {\em Introduction to Etale Cohomology}.
\newblock Springer-Verlag.

\bibitem[van Osdol, 1975]{VanOsdol}
van Osdol, D. (1975).
\newblock Homological algebra in topoi.
\newblock {\em Proceedings of the American Mathematical Society}, 50:52--54.

\bibitem[Weibel, 1994]{WeibelIntro}
Weibel, C. (1994).
\newblock {\em An introduction to homological algebra}.
\newblock Cambridge University Press.

\bibitem[Wiles, 1995]{WilesFerm}
Wiles, A. (1995).
\newblock Modular elliptic curves and {F}ermat's {L}ast {T}heorem.
\newblock {\em Annals of Mathematics}, 141:443--551.

\bibitem[Zbierski, 1971]{ZbierskiHigher}
Zbierski, P. (1971).
\newblock Models for higher order arithmetics.
\newblock {\em Bull.\ Acad.\ Pol.\ Sci.\ Ser.\ Math.\ Astron.\ Phys.},
  XIX:557--62.

\end{thebibliography}

\end{document}